\documentclass [12pt]{article}
\usepackage{amsmath,amssymb,amscd,bigfoot,amsthm}
\usepackage{titlesec}
\usepackage[margin=1in]{geometry}

\makeatletter
\def\@endtheorem{\endtrivlist}
\makeatother

\begin{document}
\newcommand{\Pnd}{{\cal P}_{n}^{d}}
\def\l{\lambda}
\def\m{\mu}
\def\a{\alpha}
\def\b{\beta}
\def\g{\gamma}
\def\G{\Gamma}
\def\d{\delta}
\def\e{\epsilon}
\def\o{\omega}
\def\O{\Omega}
\def\v{\varphi}
\def\t{\theta}
\def\r{\rho}
\def\bs{$\blacksquare$}
\def\bp{\begin{proposition}}
\def\ep{\end{proposition}}
\def\bt{\begin{theo}}
\def\et{\end{theo}}
\def\be{\begin{equation}}
\def\ee{\end{equation}}
\def\bl{\begin{lemma}}
\def\el{\end{lemma}}
\def\bc{\begin{corollary}}
\def\ec{\end{corollary}}
\def\pr{\noindent{\bf Proof: }}
\def\note{\noindent{\bf Note. }}
\def\bd{\begin{definition}}
\def\ed{\end{definition}}
\def\C{{\mathbb C}}
\def\P{{\mathbb P}}
\def\Z{{\mathbb Z}}
\def\R{{\mathbb R}}
\def\d{{\rm d}}
\def\deg{{\rm deg\,}}
\def\deg{{\rm deg\,}}
\def\arg{{\rm arg\,}}
\def\min{{\rm min\,}}
\def\max{{\rm max\,}}
\def\d{{\rm d}}
\def\deg{{\rm deg\,}}
\def\deg{{\rm deg\,}}
\def\arg{{\rm arg\,}}
\def\min{{\rm min\,}}
\def\max{{\rm max\,}}
\newcommand{\tab}{\hspace*{2em}}
\newtheorem{theo}{Theorem}[section]
\newtheorem{lemma}{Lemma}[section]
\newtheorem{definition}{Definition}[section]
\newtheorem{corollary}{Corollary}[section]
\newtheorem{proposition}{Proposition}[section]
\newtheorem{problem}{Problem}[section]

\begin{titlepage}
	\begin{center}

		\topskip 5mm
								
		{\LARGE{\bf {A case of multivariate
												
			\smallskip
												
		Birkhoff interpolation using high order derivatives}}}
								
		\vskip 6mm
								
		{\large {\bf Gil Goldman$^{*}$}}
								
		\vspace{3 mm}
	\end{center}
				
	{$^{*}$ Department of Mathematics, The Weizmann Institute of
		Science, Rehovot 76100, Israel. \ e-mail: gilgoldm@gmail.com}
				
	\vspace{3 mm}
	\begin{center}
								
		{ \bf Abstract}
	\end{center}
	{\small{We consider a specific scheme of multivariate Birkhoff polynomial interpolation. 
	Our samples are derivatives of various orders $k_j$ at fixed points $v_j$ along fixed straight lines through $v_j$ 
	in directions $u_j$, under the following assumption: the total number of sampled derivatives of order 
	$k, \ k=0,1,\ldots$ 
	is equal to the dimension of the space homogeneous polynomials of degree $k$. We show
	that this scheme is regular for general directions. 
	Specifically this scheme is regular independent of the position of the interpolation nodes.
	In the planar case, we show that this scheme is regular for distinct directions.
								
	Next we prove a ``Birkhoff-Remez'' inequality for our sampling scheme extended to larger sampling sets. 
	It bounds the norm of the interpolation polynomial through the norm of the samples, in terms of the geometry 
	of the sampling set.
	}}
				
	            \medskip
            
            \noindent {\bf Keywords} Norming set $\cdot$
            Norming constant $\cdot$ Multivariate Hermite interpolation $\cdot$ Remez-type inequality $\cdot$ Multivariate Birkhoff interpolation $\cdot$ Poised
\end{titlepage}
\newpage


\section{Introduction}\label{Sec:Intro}
	\setcounter{equation}{0}
	 
	In this paper we study a specific scheme of Birkhoff polynomial interpolation. As in many other cases 
	(compare \cite{alexander1997interpolation,ciliberto1998interpolation,eastwood1991interpolation,gasca1982lagrange}) 
	our samples are derivatives of various orders $k_j$ at fixed points $v_j$ along fixed straight lines through $v_j$ 
	in directions $u_j$. 
	It is convenient to assume that for each $j$ exactly one derivative is sampled, but the points and the directions 
	are allowed to coincide. 
	However, we make the following additional assumption: as we consider sampling of polynomials of degree $d$, 
	the total number of sampled derivatives of order $k, \ k=0,1,\ldots,d,$ 
	is equal to the dimension of the space of homogeneous polynomials of degree $k$.
	 
	We show that the necessary and sufficient conditions on $v_j,u_j$ for this Birkhoff interpolation problem to 
	be well-posed 
	take a rather simple form, they depend only on the sampling directions $u_j$, 
	{\it but are independent of the sampling points $v_j$}.
	 
	Our second result is a ``Birkhoff-Remez'' (see \cite{remez1936propriete}) inequality for the sampling scheme considered 
	(extended to larger sampling sets). 
	It bounds the norm of the interpolation polynomial on a compact set by the norm of the samples, in terms of the geometry 
	of the sampling set.
	 
	\smallskip
	 
	Our results are motivated by the following two basic questions:
	 
	\smallskip
	 
	\noindent 1. For a prescribed type of Hermite or Birkhoff interpolation problem, find 
	the conditions by which there exists a unique solution.
	This problem is central in Approximation Theory (see \cite{lorentz2000multivariate,sauer1995multivariate,bos2008multipoint} and references therein). 
	It was traditionally considered, from a somewhat different point of view, also in Algebraic geometry 
	(see \cite{miranda1999linear,alexander1997interpolation,ciliberto2001geometric,ciliberto1998interpolation,froberg2012waring} and references therein). 
	In particular, one of central open problems in this direction is the general dimensionality problem.
	Consider the following multivariate Hermite interpolation problem.
	\begin{problem}[Hermite Interpolation]\label{prob.hermite}
		Let $\{v_1,..,v_r\} \subset \R^n$ be a set of points and let $\{m_1,..,m_r\}$ be a 
	    set of positive integers such that $\sum_{k=1}^{r}\linebreak[0]\binom{m_k-1+n}{n}= \dim({\cal P}_n^d)$.
	    For each $k=1,..,r$, let $\{\psi_{\alpha,k}\}_{|\alpha| < m_k}$ be a given set of real values. 
	    Find $P \in {\cal P}_{n}^d$ that satisfies  
	    \begin{equation}\label{eq.hemite}
	            D^{\alpha}P(v_k) = \psi_{\alpha,k}, \tab |\alpha| < m_k, \tab k=1,..,r.
	    \end{equation}
	\end{problem}
	
 	 Above and forth we use the standard multi-index notation.
	 For $\alpha=(\alpha_1,..,\alpha_n),\allowbreak \alpha \in \{\mathbb{N} \cup 0\}^n$, we define:
     Absolute value, $|\alpha| =\alpha_1+..+\alpha_n$;
     Power, for $u \in \mathbb{R}^n$, $u^{\alpha} = u_1^{\alpha_1}\cdot .. \cdot u_n^{\alpha_n}$;
     Partial derivative, for $x=(x_1,\ldots,x_n) \in \mathbb{R}^n$, 
     $D^{\alpha}=\frac{\partial^{|\alpha|}}{\partial x^\alpha} = 
     \frac{\partial^{|\alpha|}}{\partial x_1^{\alpha_1}..\partial x_n^{\alpha_n}}$. 
     We also define $D^{0}P(v)=P(v)$.
     	
	 For $P \in \Pnd$ given in monomial basis, $P = \sum_{|\alpha|\le d}\a_{\alpha} x^{\alpha}$, let $A=A(v_1,..,v_r)$ be
	 the left hand side matrix form of the linear system (\ref{eq.hemite}).
	 Since the determinant of $A$ is
	 polynomial in the points $v_1,..,v_r$, $A$ is either 
	 singular for every set of points or it is regular
	 for almost all sets. Consequently we say that an interpolation problem is almost regular 
	 if it is uniquely solved for almost all sets of points.   
	 {\it The general dimensionality problem asks, for which $n,d,r,m_1,\ldots,m_r$, the corresponding problem is almost
	 regular on ${\cal P}_n^d$.} 
	 In particular, this is {\it not the case} for $n=d=r=2$ and $m_1,m_2=2$. Indeed, here the number of samples is $6$ 
	 which is the dimension of 
	 the space of quadratic polynomials. However, for any two points $v_1,v_2\in {\mathbb R}^2$ let $ax+by+c=0$ 
	 be the equation of the straight 
	 line through $v_1$ and $v_2$. Then $P(x,y)=(ax+by+c)^2$ is a nonzero polynomial of degree $2$, vanishing together 
	 with its first order 
	 derivatives both at $v_1$ and at $v_2$. See \cite{ciliberto2001geometric} for a stimulating discussion of this problem.
	 
	 The case of multivariate Birkhoff interpolation is a generalization of Hermite interpolation.  
	 We are now allowed to take directional derivatives and these need not be consecutive at each point.
	 This adds another complication to the regularity question of a specific instance of a  Birkhoff interpolation problem. In this case regularity may depend on both the points and the
	 directions.
	 
	 Many specific Hermite and Birkhoff interpolation schemes were shown to be almost regular
	 (\cite{alexander1997interpolation,ciliberto1998interpolation,eastwood1991interpolation,gasca1982lagrange,gevorgian1995bivariate,hakopian2000class,lorentz1992multivariate}).
	In our scheme regularity is achieved for {\it any choice of the points}, if the directions are generic (that is, for
	almost all sets of directions).
	 
	\smallskip
	 
	\noindent 2. The second question is to provide explicit estimates of the robustness of a certain polynomial interpolation (or reconstruction) 
	scheme. 
	This leads to ``Remez-type'' (or ``Norming'') inequalities (see \cite{brudnyi2013norming,coppersmith1992growth} and references therein). 
	Our second main result is a ``Birkhoff-Remez'' inequality for the interpolation scheme.

\section{The problem and its regularity}\label{Sec:Unisolv}
	\setcounter{equation}{0}
	 
	Let us give now an accurate setting of the problem. Denote by ${\cal P}_{n}^{d}$ the space of polynomials of 
	degree at most $d$ on $\mathbb{R}^n$ and by ${\cal L}_{n}^{d} \subset {\cal P}_{n}^{d}$ the space of 
	homogeneous polynomials 
	of degree $d$ on $\mathbb{R}^n$. The dimension of ${\cal P}_{n}^{d}$ is thus $\binom{d+n}{n}$, 
	that is, the number of distinct monomials in $n$-variables of degree at most $d$, 
	and the dimension of ${\cal L}_{n}^{d}$ is $\binom{d+n-1}{n-1}$. We define $N_{n,d}=\dim{({\cal L}_{n}^{d})}$. 
	Note that ${\cal P}_{n}^{d}$  is a direct sum of ${\cal L}_{n}^{0}, \ldots, {\cal L}_{n}^{d},$ hence 
	$\dim({\cal P}_{n}^{d})=\sum_{k=0}^d N_{n,k}.$
	 
	\smallskip
	 
	Let $P\in {\cal P}_{n}^{d}$. For a point $v$ and a direction $u$ in $\mathbb{R}^n$, we denote 
	by $D^k_uP(v)$ the $k$-th derivative $\left.\frac{d^k P(v+tu)}{dt^k}\right|_{t=0}$ of $P$ at $v$ along the straight line 
	in the direction $u$. With a slight abuse of this notation, we will also define $D^{0}_u P(v) = P(v)$.  
	\begin{problem}\label{prob.main}
		For each $k=0,1,\ldots,d$:
		let $Z_k=\{(v_{k,j},u_{k,j})\},\ j=1,\allowbreak \ldots,N_{n,k},$ be 
		a given set of pairs of a point $v_{k,j} \in \mathbb{R}^n$ and a direction vector $u_{k,j} \in \mathbb{R}^n$.
		For each $k=0,1,\ldots,d$: let $\Psi_k= \{\psi_{k,j}\}\subset \mathbb{R}, \ j=1,\ldots,N_{n,k}$, be a given set of
		real values.
		We seek a polynomial $P \in {\cal P}_{n}^{d}$ which satisfies
		\be\label{constraints}
			D^k_{u_{k,j}}P(v_{k,j})=\psi_{k,j}, \tab j=1,..,N_{n,k}, \tab k=0,..,d.
		\ee
	\end{problem} 
	
	Problem \ref{prob.main} is called regular on ${\cal P}_{n}^{d}$, given $Z_0,\ldots,Z_d$, if it has a unique 
	solution $P \in {\cal P}_{n}^{d}$, for all $\Psi_k= \{\psi_{k,j}\}\subset \mathbb{R}, \ j=1,\ldots,N_{n,k} , \
	k=0,\ldots,d.$
	
	Here is our first main result:
    \begin{theo}\label{thm.uniqueSolution}
    	For any sample points $v_{k,j}$, and for general directions $u_{k,j},$ Problem \ref{prob.main} is regular.
        Moreover, sets of directions that do not define a unique solution are exactly those for which
        there exists at least one $1 \le k \le d$ such that the directions
        $u_{k,1},..,u_{k,N_{n,k}}$ are roots of some nonzero, $n\mbox{-variate}$ homogeneous polynomial
        of degree $k$.
    \end{theo} 
    Note that Theorem \ref{thm.uniqueSolution} is independent of the configuration of the points.
    
	Proving Theorem \ref{thm.uniqueSolution}, we now consider the following intermediate problem.
	\begin{problem}\label{prob.hom}
		Let the set $\tilde Z= \{s_j\}\subset \mathbb{R}^n, \ j=1,\ldots,N_{n,k},$ and 
		$\Psi= \{\psi_{j}\}\subset \mathbb{R}, \ j=1,\ldots,N_{n,k}$, $k \in \mathbb{N}$, be given.
		We seek for $P \in {\cal L}_{n}^{k}$ that satisfies
		\be\label{constraints.11}
			P(s_{j})=\psi_{j}, \tab j=1,\ldots,N_{n,k}.
		\ee
	\end{problem}
	
	Problem \ref{prob.hom} is called regular on ${\cal L}_{n}^{k}$ given $\tilde Z$, if it has a unique solution 
	$P \in {\cal L}_{n}^{k}$ for each $\Psi$.
	 
	\smallskip
	 
	For $\tilde Z= \{s_j\}\subset \mathbb{R}^n, \ j=1,\ldots,N_{n,k},$ consider the following multidimensional homogeneous
	Vandermonde matrix $A_{n,k}(\tilde Z)$, $[A_{n,k}(\tilde Z)]_{i,j}=s_i^{\alpha_j}, \ i,j=1,\ldots, N_{n,k}$, 
	where the multi-indices $\alpha, |\alpha|=k$, are ordered lexicographically. 
	$A_{n,k}(\tilde Z)$ is the matrix associated with Problem \ref{prob.hom},
	written in the monomial basis.
	Consequently Problem \ref{prob.hom} is regular given $\tilde Z$ if and only if the determinant of $A_{n,k}(\tilde Z)$
	is nonzero.
	Since this determinant is homogeneous in the coordinates of each of the vectors $s_i$, this property depends only on the 
	directions of the vectors $s_i$, but not on their length.
	 
	\smallskip
	 
	Let us return to Problem \ref{prob.main}. For a set $Z_k=\{(v_{k,j},u_{k,j})\}, j=1,\ldots,N_{n,k}$, of pairs of
	a point $v_{k,j}$ and a direction $u_{k,j}$ in $\mathbb{R}^n$, denote by 
	$\tilde Z_k=(u_{k,1},\allowbreak \ldots,u_{k,N_{n,d}}) \subset {\mathbb R}^n$, 
	the set of the corresponding directions $u_{k,j}$.
	\begin{proposition}\label{prop.main1}
		Problem \ref{prob.main} is regular on ${\cal P}_{n}^{d}$ given $Z_0,\ldots,Z_d$, if and only if, 
		Problem \ref{prob.hom} is regular on ${\cal L}_{n}^{k}$ given $\tilde Z_k$, for each $k=1,\ldots,d$.
		Equivalently, the determinant of $A_{n,k}(\tilde Z_k)$ is nonzero for each $k=1,\ldots,d$. 
		In particular, the regularity of Problem \ref{prob.main} is determined only by the directions 
		of the vectors $u_{k,j}$, and is invariant with respect to their length, and with respect to the position of the 
		points $v_{k,j}$.
	\end{proposition}
	
	\begin{proof}
		For each polynomial 
		$P=\sum_{|\alpha|\le d}a_\alpha x^\alpha \in {\cal P}_{n}^{d},$ 
		and for each $k=0,1,\ldots,d$, denote by $P_k=\sum_{|\alpha|=k}a_\alpha x^\alpha$ the $k$-th homogeneous 
		component of $P$, 
		and put $\tilde P_k=\sum_{l=k}^d P_l$. Let us recall that for a point $v$ and a direction $u$ in $\mathbb{R}^n$, 
		we denote by $D^k_uP(v)$ the $k$-th derivative $\left.\frac{d^k P(v+tu)}{dt^k}\right|_{t=0}$ of $P$ at $v$ along 
		the straight line in the direction $u$.
		 
		\bl\label{lem:deriv.homog}
		For $P,P_k,\tilde P_k$ as above, for each $v,u\in {\mathbb R}^n,$ and for each $k=1,\ldots,d,$ we have
		 
		\be\label{eq:deriv.homog}
		D^k_uP_l(v)=0, \ l<k, \ \ D^k_uP(v)=D^k_u\tilde P_k(v), \ \  D^k_uP_k(v)=k!P_k(u).
		\ee
		\el
		
		\begin{proof}
			The first equality is immediate, and the second follows directly from the first. The third one is Euler's 
			identity for homogeneous polynomials. It follows directly from the fact that for $P(x)$ being the 
			monomial $x^\alpha,       \ |\alpha|=k,$ the highest $k$-th degree term in 
			$P(v+tu)$ is $u^\alpha\cdot t^k.$
		\end{proof}
		 
		\medskip
		 
		Assume now that Problem \ref{prob.hom} is regular on ${\cal L}_{n}^{k}$ given $\tilde Z_k$, for each
		$k=1,\ldots,d$.
		Consider the part of the interpolation equations (\ref{constraints}) for Problem \ref{prob.main} with the highest 
		order derivatives:
		 
		\begin{equation}\label{eq.hom}
			D^d_{u_{d,j}}P(v_{d,j})=\psi_{d,j}, \tab j=1,..,N_{n,d}.
		\end{equation}
		By Lemma \ref{lem:deriv.homog} these equations are reduced to $P_d(u_{d,j})=\frac{1}{d!}\psi_{d,j}, \ j=1,..,N_{n,d}$. 
		This is an instance of Problem \ref{prob.hom}, and by our assumption this system is regular. Hence, the highest
		homogeneous component $P_d$ of a solution $P$, is uniquely determined by (\ref{eq.hom}). Next we consider $\hat P_d=P-P_d$. 
		This is a polynomial of degree $d-1$, and it satisfies the corrected system of equations (\ref{constraints}), 
		which for the derivatives of order $d-1$ takes the form:
		$$
			D^{d-1}_{u_{d-1,j}}\hat{P}_d(v_{d-1,j})=\psi_{d-1,j}-D^{d-1}_{u_{d-1,j}}P_d(v_{d-1,j}),\tab j=1,..,N_{n,d-1}.
		$$
		As above, from this system we find the unique homogeneous component $P_{d-1}$ of $P$. 
		Continuing in this way till the degree one and then recovering the constant term of $P$ by setting $P_{0}(v_{0,1}) =  \psi_{0,1}-\sum_{k=1}^d P_{k}(v_{0,1})$,
		we have uniquely reconstructed $P$.
		 
		\medskip
		 
		In the opposite direction, assume that given Problem \ref{prob.main} with the sets $Z_1,\ldots,Z_d$, some of the
		associated Problems \ref{prob.hom} are not regular. Fix the smallest index $l\le d$ for which this happens. Then,
		we can find a nonzero homogeneous polynomial $P_l$ such that
		$$
			D^l_{u_{l,j}}P_l(v_{l,j})=0, \tab j=1,..,N_{n,l}.
		$$
		Now we construct the right hand side in Problem \ref{prob.main} for which it has a nonzero solutions. 
		Start with the solution $P=P_l$ and put
		 
		$$
			\psi_{k,j}:= D^k_{u_{k,j}}P_l(v_{k,j}), \tab j=1,..,N_{n,k}, \tab k=0,..,d.
		$$
		By the construction, $\psi_{k,j}=0$ for  $k \ge l$. However, $\psi_{k,j}$ may be 
		nonzero for $k=0,1,\ldots,l-1$. We consider now Problem \ref{prob.main} on the sets $Z_0,Z_1,\ldots,Z_{l-1}$ 
		for ${\cal P}_{n}^{l-1},$ that is, for polynomials of degree  $l-1$. Since by construction Problem \ref{prob.hom} is 
		regular on each of the sets $\tilde Z_0, \tilde Z_1,\ldots, \tilde Z_{l-1},$ we conclude, by the already 
		proved part of Proposition \ref{prop.main1}, that Problem \ref{prob.main} is regular for  
		${\cal P}_{n}^{l-1}$ on $Z_0, Z_1,\ldots, Z_{l-1}.$ 
		So we can find (uniquely) a polynomial $P'$ of degree $l-1$, such that
		 
		$$
			D^k_{u_{kj}}P'(v_{k,j})=\psi_{k,j}, \tab j=1,..,N_{n,k}, \tab k=0,..,l-1.
		$$
		Therefore, $P=P_l-P'$ is a nonzero polynomial such that
		 
		$$
			D^k_{u_{kj}}P(v_{k,j})=0, \tab j=1,..,N_{n,k}, \tab k=0,..,d.
		$$
		We conclude that Problem \ref{prob.main} on $Z_0, Z_1,\ldots, Z_d$ is not regular. This completes the proof of
		Proposition \ref{prop.main1}. 
	\end{proof}
	
	We now prove Theorem \ref{thm.uniqueSolution}. 
	\begin{proof}
		By Proposition \ref{prop.main1}, the regularity of Problem \ref{prob.main} on $\Pnd$ given the sets
		$Z_0, Z_1,\ldots, Z_d$, is equivalent to the regularity of Problem \ref{prob.hom} on ${\cal L}_{n}^k$ given
		the set $\tilde{Z}_k = (u_{k,1},\ldots,u_{k,N_{n,k}})$, for each $k=1,\ldots, d$. 
		In turn, for each $k=1,\ldots,d$, Problem \ref{prob.hom} on ${\cal L}_{n}^k$ given 
		$\tilde{Z}_k = (u_{k,1},\ldots,u_{k,N_{n,k}})$, is the standard 
		Lagrange interpolation on ${\cal L}_{n}^k$ which is regular exactly when
		$(u_{k,1},\ldots,u_{k,N_{n,k}}) \subset \R^n$ are not the roots of some $P \in {\cal L}_{n}^k$. 
	\end{proof}		 
	\bt\label{thm.2d}
		In the planar case, for $n=2$, Problem \ref{prob.main} is regular if and only if for each $k=1,\ldots,d$, the
		directions of the vectors $u_{k,j}$ are pairwise linearly independent.
	\et
	
	\begin{proof}
	In the planar case, for each $k$, it is known that (but might not be easy to locate reference to, see for example
	\cite{yaacov2014multivariate}) the homogeneous Vandermonde determinants of the matrix $A_{2,k}(\tilde Z_k)$ take
	the following convenient form:
	$$det \left( A_{2,k}(\tilde Z_k) \right)= \prod_{1 \le i < j \le k+1}det[u_{k,i},u_{k,j}],$$ 
    where $det[u_{k,i},u_{k,j}]$ denotes the determinant of the two by two matrix having $u_{k,i}$ and $u_{k,j}$
    as its rows. It follows that $det \left( A_{2,k}(\tilde Z_k) \right)$ is nonzero is equivalent to the directions
    of $k^{th}$ Homogeneous Problem being pairwise independent. Finally, by Proposition \ref{prop.main1}, the regularity
    of the homogeneous problems is equivalent to the regularity of Problem  \ref{prob.main}.
    \end{proof}

\section{Birkhoff-Remez inequality for Problem \ref{prob.main}}\label{Sec:Remez}
	\setcounter{equation}{0}
	 
	\subsection{Norming inequalities}\label{Sec:Remez1}
		 	
		The classical Remez inequality and its generalizations compare maxima of a polynomial $P$ on two sets 
		$U\subset G$ (see \cite{brudnyui1973extremal,brudnyi2013norming,remez1936propriete,yomdin2011remez} and references therein). We would like to extend this 
		setting, in order to include into sampling information on $U$ the derivatives of $P$. 
		Such an extension is provided by a wider (and also classical) setting of ``norming sets'' 
		and ``norming inequalities''. Let $G\subset {\mathbb R}^n$ be a compact domain, and 
		let $L\subset {\cal P}_n^d$ be a normed linear subspace of the space of polynomials of degree at most $d$, 
		equipped with the norm $||P||_G=\max_{G}|P|$. Let $L^*$ denote the dual space of all the linear functionals on $L$.
		 	
		\bd\label{definite.lin}
		Let $U \subset L^*$ be a bounded set of linear functionals on $L$. A semi-norm $||P||_U$ on $L$ is defined as
		\be\label{lin.funct.norm}
		||P||_U = \sup_{w\in U} |w(P)|.
		\ee
		The set $U$ is said to be $L$-norming if the semi-norm $||P||_U$ is in fact a norm on $L$, that is, if for $P\in L$ 
		we have that if $||P||_U=0$ then $P=0$. Equivalently, $U$ is $L$-norming if
		 	
		\be\label{FRemez.lin}
		N_{L}(G,U):= \sup_{P\in L, P\ne 0}\frac{||P||_G}{||P||_U}< \infty.
		\ee
		$N_{L}(G,U)$ is called the $L$-norming constant of $U$ on $G$.
		\ed
		The usual Remez-type inequalities are included in the new setting by identifying $x\in U\subset G$ with the 
		linear functional $\delta_x$ sampling a polynomial at the point  $x$.
		 	
		\smallskip
		 	
	\subsection{Robust polynomial reconstruction}\label{Sec:Remez2}
		 	
		To illustrate the role of the norming constant in estimating the robustness of polynomial reconstruction, 
		let us consider the following reconstruction scheme: our goal is to find the ``best'' polynomial approximation 
		of a given function $f$ on $G$ according to the norm $||\cdot||_G$. In other words, the ``ideal approximation'' 
		is a polynomial $\bar P= \arg \min _{P\in {\cal P}_n^d}||f-P||_G$ (where $\arg \min_{P\in {\cal P}_n^d}$ 
		is the operator extracting the minimizing polynomial).
		Let us denote the error
		$||f-\bar P||_G$ of this ideal approximation by $E$. Now, let us assume that our input consists of noisy measurements of $f$ 
		on a subset $U\subset G$ (this setting can be easily extended to the 
		measurements being more general linear functionals). 
		So we start with a function $\tilde f=f+\nu$ on $U$, where $\nu$ is the measurement error function, satisfying 
		$\max_U |\nu(x)|\leq h$. 
		As an output we take $\hat P= \arg \min _{P\in {\cal P}_n^d}||\tilde f-P||_U$. The following result 
		shows that the output error $||\hat P-\bar P||_G$ can be bounded in terms of the norming 
		constant $N=N_{{\cal P}_n^d}(G,U)$, $E$ and $h$.
		 	
		\bp\label{prop:robust}
		 	
		$$
			||\hat P-\bar P||_G \leq 2N(E+h).
		$$
		\ep
		
		\begin{proof}
		By the construction of $\bar P$ we have that
		$||\tilde f-\bar P||_U\le ||f-\bar P||_U + ||\nu||_U \le E+h.$ Since 
		$\hat P= \arg \min _{P\in {\cal P}_n^d}||\tilde f-P||_U$, then
		$||\tilde f-\hat P||_U \le ||\tilde f-\bar P||_U \le  E+h$.
		We conclude that $||\bar P-\hat P||_U \le 2(E+h),$ 
		and hence $||\bar P-\hat P||_G \le 2N(E+h).$
		\end{proof}
		 	
	\subsection{Norming inequality for extended Problem \ref{prob.main}}\label{Sec:Remez3}
		 	
		Let us recall (and extend) some notations introduced in the proof of Theorem \ref{thm.uniqueSolution}. 
		For each polynomial $P=\sum_{|\alpha|\le d}a_\alpha x^\alpha \in {\cal P}_{n}^{d},$ 
		and for each $k=0,1,\ldots,d$, we have denoted by $P_k=\sum_{|\alpha|=k}a_\alpha x^\alpha$ 
		the $k$-th homogeneous component of $P$. We will denote by ${\cal P}^{d,k}_n\subset {\cal P}_n^d$  
		the subspace of all the polynomials $P$ of the form $P=\sum_{k \le |\alpha|\le d}a_\alpha x^\alpha$. 
		Recall that for $P\in {\cal P}_{n}^{d}$ we have denoted $\tilde P_k=\sum_{l=k}^d P_l\in {\cal P}^{d,k}_n$.
		 	
		\smallskip
		 	
		Returning to Problem \ref{prob.main}, we now extend its setting, allowing larger sets $Z_k$. 
		So now $Z_k\subset {\mathbb R}^n\times {\mathbb R}^n$ may be any bounded set of couples $(v,u)$ 
		of a point and a direction vector.
		 	
		For any linear subspace $L\subset {\cal P}_n^d$ the sets $Z_k$ can be considered as 
		subsets $\bar Z_k\subset L^*$, if we identify the couple $(v,u)\in Z_k$ with the linear 
		functional $D^k_{u,v}$ on $L$ defined by $D^k_{u,v}(P)=D^k_{u}P(v).$ For the sampling sets
		$Z_0,\ldots,Z_d$, we define $U=U(Z_0,\ldots,Z_d)=\cup_{k=0}^d \bar Z_k \subset L^*$,
		and $U_k=\cup_{l=k}^d \bar Z_l \subset L^*$.
		 	
		On the other hand, extending the notations used in Theorem \ref{thm.uniqueSolution}, 
		we denote by $\tilde Z_k=\{u: \exists v, \ (v,u)\in Z_k\} \subset {\mathbb R}^n $ 
		the set of the directions $u$ that appear in $Z_k$. As above, 
		for any linear subspace $L\subset {\cal P}_n^d$ the set $\tilde Z_k$ 
		can be considered as a subset $\tilde Z_k\subset L^*$, via identifying $u\in \tilde Z_k$ 
		with the evaluation functional at the point $u$, $\delta_u \in L^*$.
		 	
		\smallskip
		 	
		To simplify the presentation we fix $G$ to be equal to the unit ball $B=B^n_1\subset {\mathbb R}^n$, 
		and assume that for each $k=0,\ldots,d$, the sets $Z_k$ satisfy $Z_k \subset B\times B$, that is, 
		both the sampling points $v$ and the directions $u$ belong to the unit ball $B$.
		 	
		\bt\label{thm:main2}
		For each $k=0,\ldots,d$, set $L_k={\cal L}_{n}^k\subset {\cal P}_n^d$ to be the subspace of 
		homogeneous polynomials of degree $k$. Let each of the directions sets $\tilde Z_k\subset L_k^*$ 
		be $L_k$-norming on $B$, with norming constant $\theta_k=N_{L_k}(B,\tilde Z_k)$. 
		Then $U=U(Z_0,\ldots,Z_d)$ is norming for ${\cal P}_n^d$, with the 
		norming constant $N_{{\cal P}_n^d}(B,U)$ satisfying
		 	
		$$
			N_{{\cal P}_n^d}(B,U)\leq \sum_{l=0}^d \frac{\theta_l}{l!} \cdot \prod_{j=0}^{l-1} (1+ m_j\cdot \frac{\theta_j}{j!}),
		$$
		where $m_l=m(d,l)=T_d^{(l)}(1)$. Here $T_d(x)$ is the $d$-th Chebyshev polynomial, 
		and $T_d^{(k)}(x)$ is its $k$-th derivative.
		\et
		
		\begin{proof}
		Let us denote by $\kappa_k$ the norming constant $N_{L_k}(B,\bar Z_k)$. Then for each $k=0,\ldots,d$, we have
		 	
		\be\label{eq:factorial}
			\kappa_k=\frac{\theta_k}{k!}
		\ee
		Indeed, by Lemma \ref{lem:deriv.homog}, the sets of linear functionals on $L_k$, $\tilde Z_k$ 
		and $\bar Z_k$ satisfy $\bar Z_k=k!\tilde Z_k$.
		 	
		\smallskip
		 	
		Next we prove by induction (starting from $k=d$ and going down) the following result:
		 	
		\bl\label{lem:induction}
		For each $k=0,\ldots,d$, the set $U_k$ is norming for ${\cal P}^{d,k}_n$ on $B$. 
		The norming constant $\eta_k=N_{{\cal P}^{d,k}_n}(B,U_k)$ satisfies
		
		\begin{align}\label{eq:indiction}
			\begin{split}
				\eta_d &= \kappa_d,\\ 
				\eta_k &\le \kappa_k+(1+ m_k\cdot \kappa_k)\eta_{k+1}, \tab k<d.
			\end{split}
		\end{align}
		\el
		
		\begin{proof}
		For $k=d$ our problem is reduced to Problem \ref{prob.hom} on the space $L_d={\cal L}_{n}^d\subset {\cal P}_n^d$ 
		of homogeneous polynomials of degree $d$. By assumptions, and by (\ref{eq:factorial}), 
		we have $N_{L_d}(B,U_d)=N_{L_d}(B,\bar Z_d)=\kappa_d=\frac{\theta_d}{d!}>0$. 
		Consequently Lemma \ref{lem:induction} holds for the case $k=d$. 
		Assume that (\ref{eq:indiction}) is satisfied for $k=l+1\leq d$ and prove it for $k=l$. Let $P\in {\cal P}^{d,l}_n$ 
		satisfy $|w(P)|\le 1$ for each $w\in U_l$. In particular, $|w(P)|\le 1$ for all $w\in U_{l+1}$. 
		Since for $w\in U_{l+1}$, $w(P)=w(\tilde P_{l+1})$, we have that for $w\in U_{l+1}$, 
		$|w(\tilde P_{l+1})|\le 1$.  By the induction assumption, 
		and by definition of the norming constant we have that
		 	
		\be\label{eq:norm.tilde}
		||\tilde P_{l+1}||_B\leq \eta_{l+1}.
		\ee
		 	
		Now, for the homogeneous component $P_l$ of $P$ we have $P_l=P-\tilde P_{l+1}$, 
		and thus for $w\in \bar Z_l$, $|w(P_l)|=|w(P-\tilde P_{l+1})|\leq 1+ |w(\tilde P_{l+1})|$. 
		To estimate the values $w(\tilde P_{l+1})$, which are the directional derivatives of order $l$ 
		of $\tilde P_{l+1}$, we apply the classical Markov inequality, in the form presented in \cite{harris2008multivariate,skalyga2005bounds}:
		 	
		\bt\label{thm:Markov}
		For $P \in {\cal P}_{n}^d$, and for any direction vector $u\in {\mathbb R}^n, \ ||u||\leq 1$,
		 					
		$$
		||D^k_{u}P||_B \le m_k||P||_B,
		$$
		where $m_k=T_d^{(k)}(1)$.
		\et
		
		Applying this result to $\tilde P_{l+1}$ we conclude, using (\ref{eq:norm.tilde}) 
		that for $w\in \bar{Z}_l$ the bound $|w(\tilde P_{l+1})|\leq m_l\cdot\eta_{l+1}$ holds. 
		Therefore, for $w\in \bar Z_l$ we get $|w(P_l)|\leq 1+ m_l\cdot \eta_{l+1}$.
		 	
		By the assumptions of the theorem the set $\tilde Z_l$ is norming for $L_l$, 
		with the norming constant $N_{L_l}(B,\tilde Z_k)=\theta_l.$ Therefore, by (\ref{eq:factorial}), 
		the set $\bar Z_l$ is also norming for $L_l$, with the norming constant 
		$\kappa_l=N_{L_l}(B,\bar Z_k)=\frac{\theta_l}{k!}.$  
		We conclude that		
		$$||P_l||_B\le \kappa_l(1+ m_l\cdot \eta_{l+1}).$$ 
		Finally, since $\tilde P_l=P_l+\tilde P_{l+1},$ we obtain
		 	
		$$
		||P||_B\leq ||\tilde P_{l+1}||_B + ||P_l||_B \leq \eta_{l+1}+ \kappa_l(1+ m_l\cdot \eta_{l+1}).
		$$
		This inequality is true for each polynomial $P\in {\cal P}^{d,l}_n$, and hence
		 	
		$$
		\eta_l\le \eta_{l+1}+ \kappa_l(1+ m_l\cdot \eta_{l+1})=\kappa_l +  (1+ m_l\cdot \kappa_l)\eta_{l+1}.
		$$
		This completes the proof of Lemma \ref{lem:induction}. 
		\end{proof}

		To complete the proof of Theorem \ref{thm:main2} it remains to solve explicitly the recurrence inequality 
		(\ref{eq:indiction}).
		 	
		\bl\label{lem:recur}
		Let $\tau_k, \ k=0,\ldots, d,$ satisfy recurrence relation 	
		\begin{align}\label{eq:recur}
			\begin{split}
				\tau_d &= \kappa_d, \\
				\tau_k &= \kappa_k +  (1+ m_k\cdot \kappa_k)\tau_{k+1}, \tab k<d.
			\end{split}
		\end{align}
		Then for each $k=0,\ldots,d$, we have
		 	
		\be\label{eq:recur1}
		\eta_k \le \tau_k = \sum_{l=k}^d \kappa_l \cdot \prod_{j=k}^{l-1} (1+ m_j\cdot \kappa_j).
		\ee
		Here the empty product (for $l=k$) is assumed to be equal to one.
		\el
		
		\begin{proof}
		First we prove by induction the expression for $\tau_k$. For $k=d$ we have $\tau_d=\kappa_d$, 
		which is the right hand side of (\ref{eq:recur1}). Now, for $k<d$, using induction assumption, we have
		 	
		$$
		\tau_k = \kappa_k +  (1+ m_k\cdot \kappa_k)\tau_{k+1} = \kappa_k +(1+ m_k\cdot \kappa_k)\sum_{l=k+1}^d \kappa_l 
		\cdot \prod_{j=k+1}^{l-1} (1+ m_j\cdot \kappa_j)
		$$
		$$
		=\kappa_k +\sum_{l=k+1}^d \kappa_l \cdot \prod_{j=k}^{l-1} (1+ m_j\cdot \kappa_j)=\sum_{l=k}^d \kappa_l \cdot 
		\prod_{j=k}^{l-1} (1+ m_j\cdot \kappa_j).
		$$
		This completes the proof of Lemma \ref{lem:recur}. 
		\end{proof} 	
		 	
		To complete the proof of Theorem \ref{thm:main2} it remains to substitute into (\ref{eq:recur1}) 
		the values $\kappa_k=\frac{\theta_k}{k!}$.
		\end{proof}
		
		In the proof of theorem \ref{thm:main2} we applied Markov inequality (\ref{thm:Markov}), to upper bound 
		the $l^{th}-1$ derivatives of polynomials $P\in {\cal P}^{d,l}_n$, $l=1,\ldots,d$. These are incomplete polynomials
		and one may expect sharper Markov inequalities. Indeed approximation with incomplete polynomials and, in
		general, Markov-type inequalities for constrained polynomials are a subject of research in both approximation theory and
		general analysis (for the univariate case see, for example,
		\cite{erdelyi1998markov,cheney1980approximation,baishanski1984incomplete}).
		For the multivariate case, we are unaware of a general result which improves the upper bound in (\ref{thm:Markov}). We
		suggest here that Theorem \ref{thm:main2} can be improved using such a result.
		 
		\subsubsection{Birkhoff-Remez inequality}
			 		
			The classical Remez inequality \cite{remez1936propriete} bounds the maximum of a univariate polynomial $P$ 
			on the interval $[-1,1]$ through its maximum on a subset $Z\subset [-1,1]$ of a positive measure. 
			This theorem was extended to several variables in \cite{brudnyui1973extremal}, and then further generalized in 
			\cite{yomdin2011remez}, 
			where the Lebesgue measure $\mu_n(Z)$ was replaced by a certain quantity $\o_{n,d}(Z),$ 
			expressed through the metric entropy of $Z$:
			 		
			\bt\label{thm:remez}(\cite{yomdin2011remez}, Theorem 2.3.)
			If $\omega_{n,d}(Z)=\omega > 0$, then for each polynomial $P\in {\cal P}^d_n$
			 		
			$$	
				\sup_{x \in B}|P(x)| \le T_d \left(\frac{1+(1-\omega)^{\frac{1}{n}}}{1-(1-\omega)^{\frac{1}{n}}} \right) 
				\sup_{x\in Z}|P(x)|.
			$$	
			\et	
						
			For any measurable $Z$ we have $\o_{n,d}(Z)\geq \mu_n(Z),$ and $\o_{n,d}(Z)$ may be positive for discrete 
			and finite $Z$. We will not give here an accurate definition of $\o_{n,d}(Z),$ referring the reader to 
			\cite{yomdin2011remez}. 
			Replacing in Theorem \ref{thm:remez} $\o_{d,n}(Z)$ with a smaller value $\mu_n(Z)$ we obtain the result 
			of \cite{brudnyui1973extremal}, 
			and putting $n=1$ we get back the classical Remez inequality.
			 		
			\smallskip
			 		
			Theorem \ref{thm:main2} reduces estimation of the norming constant in our setting of Birkhoff 
			interpolation to 
			the norming constants $\theta_k$ of the direction sets $\tilde Z_k\subset {\mathbb R}^n$, 
			in the spaces of homogeneous polynomials of degree $k$. In this situation the Remez-type inequality 
			of Theorem \ref{thm:main2} is applicable. We obtain the following bound:
			 		
			\bc\label{cor:main3}
			Assume that for each $k=0,\ldots,d$ we have $\o_k=\o_{n,k}(\tilde Z_k)>0$. 
			Then we can put in Theorem \ref{thm:main2} the values 
			$\theta_k=T_d \left(\frac{1+(1-\omega_k)^{\frac{1}{n}}}{1-(1-\omega_k)^{\frac{1}{n}}} \right).$
			\ec
			
		\subsection{The univariate case}
			We now shortly discuss aspects of the reconstruction for the most basic case of $n=1$ and
			the domain of approximation being $G=[0,1] \subseteq \R$. 
			In this case, and if the number of measurements is equal to $d+1$, 
			the approximation scheme is reduced to an instance of a single-variate (classic) Birkhoff interpolation. 
			For each $k=0,\ldots,d$, we have a single measurement of the $k^{th}$ derivative of $P$ at the point 
			$v_k$ (where the points are not necessarily distinct or ordered). As before, $U=\{w_1,\ldots,w_{d+1}\}$ is the set of
			the corresponding linear functionals, $w_k(P) = \frac{d^{k-1}}{x^{k-1}}P(v_{k})$. 
			
			If the points all coincide at a point $v$, then the interpolant polynomial is given by the Taylor approximation 
			at $v$. In this case, a direct calculation shows that the norming constant $N=N_{{\cal P}_1^d}([0,1],U)$, 
			is bounded by a constant not depending on $d$. Our preliminary calculation and numerical experiments indicate 
			that $N$ remains bounded for equidistant points in $[0,1]$. 
			
			
			Another natural problem which can be considered is the accuracy of the approximation when the function to be approximated, $f$,  
			is smooth to some degree. In this setting, for a set of points $v_1,\ldots,v_{d+1}$, and for $f \in C^d[0,1]$ 
			our measurements are $\frac{d^{k-1}}{x^{k-1}}f(v_{k})$, $k=0,\ldots,d$.
			In general, the reconstruction error of the scheme will depend on the position of the points and 
			can be studied using Birkhoff reminder Theorem(see \cite{birkhoff1906general}).
			Specifically, assume we start with a sequence of points $v_1,\ldots,v_{d+1}$ and then randomly permute them 
			to get a new sequence $\tilde{v_1},\ldots,\tilde{v}_{d+1}$ made out from the same set of points. 
			Interpolating f with the permuted sequence of measurments 
			may have a significant effect on accuracy of the reconstruction comparing to 
			the non permuted sequence.
			
			We consider it an interesting and challenging problem attaining bounds for certain configurations of the points for 
			the aforementioned problem. We would like to thank the referee for these observations and for suggesting this direction.
			
%

			\vskip1cm

			    \bibliographystyle{plain} 
    			\bibliography{bib}{}

\end{document}